\documentclass[12pt]{amsart}

\usepackage[margin=1.3in, centering]{geometry}

\usepackage{enumerate}

\usepackage[utf8]{inputenc}
\usepackage{amssymb,amscd,amsmath,enumerate,amsthm}
\usepackage[all]{xy}
\usepackage[utf8]{inputenc}
\usepackage[english]{babel}
\usepackage{tikz}
\usetikzlibrary{positioning,chains,fit,shapes,calc}
\usepackage{mathtools}
\usepackage[colorlinks,linkcolor=blue]{hyperref}
\usepackage{ytableau}

\theoremstyle{plain}
\newtheorem{thm}{Theorem}[section]

\newtheorem{cor}[thm]{Corollary}

\newtheorem{prop}[thm]{Proposition}

\newtheorem{lem}[thm]{Lemma}

\theoremstyle{remark}
\newtheorem{remark}{\bf \quad \itshape  Remark}

\theoremstyle{plain}

\theoremstyle{definition}

\renewcommand{\bar}{\overline}

\newcommand{\C}{{\mathbb{C}}}

\newcommand{\Z}{{\mathbb{Z}}}

\newcommand{\GL}{\mathrm{GL}}

\newcommand{\Alt}{{\raise 2pt\hbox{$\scriptstyle\bigwedge$}}}

\definecolor{myblue}{RGB}{80,80,160}
\definecolor{mygreen}{RGB}{80,160,80}
\newdimen\nodeSize
\newdimen\nodeDist
\tikzset{
	position/.style args={#1:#2 from #3}{
		at=(#3.#1), anchor=#1+180, shift=(#1:#2)
	}
}

\allowdisplaybreaks

\title{Symmetric permutation invariants in some tensor products}

\author{Zhipeng Lu}

\address{Mathematisches Institut,
	Georg-August Universit\"{a}t G\"{o}ttingen,
	Bunsenstra{\ss}e 3-5,
	D-37073 G\"{o}ttingen,
	Germany}

\email{zhipeng.lu@uni-goettingen.de}


\date{}

\begin{document}
	
	\maketitle

	\begin{abstract}
		This is a note for constructing fundamental invariants and computing the Hilbert series of the invariant subalgebras of tensor products of polynomial rings under the action by a direct product of symmetric groups. Our computation relies on Schur functions bringing together several identities of combinatorial generating functions including that of plane partitions.
	\end{abstract}

	\section{Introduction}
	Let $K$ be any field of characteristic zero and $X$ an affine variety over $K$. Suppose $G$ is a linear algebraic group acting regularly on $X$. Then we can define an action of $G$ on the coordinate ring $K[X]$ via
	\[g\cdot f(x):=f(g^{-1}\cdot x),\  \forall g\in G, x\in X, f\in K[X].\] 
	Invariant theory generally concerns about the invariant subalgebra $K[X]^G$, consisting of functions of $K[X]$ invariant under the action of $G$. The most basic example for $X$ being a vector space $V\simeq K^n$ with the symmetric group $S_n$ permuting a canonical basis of $V$ is well furnished with the theory of symmetric functions. Generalized to $K[V_1\oplus\cdots\oplus V_k]^{S_{n_1}\times\cdots\times S_{n_k}}$ with $V_i$'s vectors spaces of dimension $n_i$ and $S_{n_i}$ acting on $V_i$, we have the following generalization of the fundamental theorem of symmetric polynomials (see Theorem 3.10.1 of Derksen-Kemper \cite{Derksen-Kemper}):
	for $R:=K[x_{ij}]_{1\leq i\leq k, 1\leq j\leq n_i}$ with $G=S_{n_1}\times\cdots\times S_{n_k}$ acting by $(\sigma_1,\dots,\sigma_k)\cdot x_{ij}=x_{i\sigma_i(j)}$, $R^G\simeq K[s_{ij}]_{1\leq i\leq k, 1\leq j\leq n_i}$, where
	\[s_{ij}:=\sum_{I\subset\{1,\dots, n_i\}, |I|=j}\prod_{l\in I}x_{il}.\]
	The result holds for field of any characteristics.
	
	In general, the invariant subalgebra of finite groups acting on polynomial rings is not a polynomial ring on any set of invariants, but a finite free module over some polynomial ring of homogeneous invariants. This is due to the Cohen-Macaulay property of finite group actions. The free module expression is usually called a \textit{Hironaka decomposition}, while a set of homogeneous parameters of the base polynomial ring are call \textit{primary invariants} and a basis of invariants of the free module over the polynomial ring are called \textit{secondary invariants}. It is of special significance in computational algebraic geometry to further find generic basis invariants of the free modules for permutation groups. 
	
	To prelude results in more details, let $G\leq S_n$ be a permutation group acting on $K[V]=K[x_1,\dots,x_n]$ by permuting indexes of the variables $x_i$ for $V$ the corresponding affine space. Then the invariant subalgebra $K[V]^G$ consisting of polynomials invariant under the action of $G$, is a free finite module over $K[V]^{S_n}$ which is a polynomial ring in the elementary symmetric polynomials. The remarkable work of Garsia-Stanton \cite{Garsia-Stanton} provides an explicit combinatorial construction of secondary invariants for a large class of important permutation groups including the Young subgroups. 
	
	In more generality, we may extend the underline polynomial ring to products of symmetric polynomial rings, i.e. $K[V_1\times\cdots\times V_k]^{S_{n_1}\times\cdots\times S_{n_k}}\simeq K[V_1]^{S_{n_1}}\otimes\cdots\otimes  K[V_k]^{S_{n_k}}$. More precisely, for any permutation subgroup $G\leq S_{n_1}\times\cdots\times S_{n_k}$, we study the Hironaka decomposition of $K[V_1\times\cdots\times V_{n_k}]^{G}$. Then we can not choose the elementary symmetric polynomials of $K[V_1\times\cdots\times V_{n_k}]$ as primary invariants, but rather the collection of separated elementary symmetric functions of $K[V_i]$'s. This is accounted in the difference of denominators of in their Hilbert-Molien series. 
	
	In this paper, we focus on the basic case where $S_n\simeq G\leq S_n\times S_n$ acting on $K[V\times V]$ diagonally, i.e. $\sigma\cdot x_i\otimes x_j=x_{\sigma(i)}\otimes x_{\sigma(j)}$ for any $\sigma\in S_n$ in tensor notation. In this case, constructing fundamental invariants that generate the invariant subalgebra follows from standard analogue of Newton's identities. Then to decode the information about secondary invariants, we compute its Hilbert series. Finding a full set of explicit secondary invariants turns out to be much more difficult and will be left for future studies. 
	
	\section{A standard system of fundamental invariants}
    We work with a more general setting than preluded in the introduction as follows:
	
	to make notations shorter, let $\Gamma:=\Gamma_1\times\cdots\times\Gamma_k$ be a Cartesian product of finite sets $\Gamma_i$ with $|\Gamma_i|=n_i$, and let $S_{n_i}$ be the symmetric group on $\Gamma_i$. Then $S_\Gamma=S_{n_1}\times\cdots\times S_{n_k}$ can be seen as the symmetric group on $\Gamma$. With a field $K$ of characteristic zero fixed, $K[V_\Gamma]=K[x_{il}]$ ($1\leq i\leq k, 1\leq l\leq n_i$) is short for $K[x_{11},\dots,x_{1n_1},\dots,x_{k1},\dots,x_{kn_k}]$. Set $V_\Gamma=V_1\times\cdots\times V_k$ with $V_i$ the standard permutation representation of $S_{n_i}$ by linearly extending its action on $\Gamma_i$. Setting $V_\Gamma=V_{1}\oplus\cdots\oplus V_{k}$ with a vector space structure would not make any difference on the coordinate ring, hence we do not distinguish the two products. $S_\Gamma$ acts on $K[x_{il}]\otimes K[x_{il}]$ via $\sigma\cdot1\otimes x_{il}=1\otimes x_{i\sigma_i(i)}$,  $\sigma\cdot x_{il}\otimes1=x_{i\sigma_i(i)}\otimes1$; $(K[V_\Gamma]\otimes K[V_\Gamma])^{S_\Gamma}=\{f\in K[V_\Gamma]\otimes K[V_\Gamma]\mid \sigma\cdot f=f, \forall\sigma\in G\}$. The tensor product $\otimes$ always denotes $\otimes_K$. We may equate $K[V_\Gamma]\otimes K[V_\Gamma]$ with $K[x_{il}][y_{il}]$ by identifying $x_{il}\otimes 1$ with $x_{il}$ and $1\otimes x_{il}$ with $y_{il}$, but we will stick to the tensor product notations.
	
	We start with introducing the basic structural theorem of invariant subalgebras for finite groups as follows.
	\begin{prop}[Hochster-Eagon \cite{Hochster-Eagon}]\label{prop-Hochster-Eagon}
	Let $G$ be any finite group acting on an affine space $X$ over a field $K$. If the characteristic of $K$ does not divide the group order $|G|$, then $K[X]^G$ is Cohen-Macaulay.		
	\end{prop}
	This tells us that $K[X]^G$ is always a finite free module over some polynomial ring $K[f_1,\dots, f_r]$ (e.g. see Proposition 2.5.3 of Derksen-Kemper \cite{Derksen-Kemper}) for a \textit{ homogeneous system of parameters} $f_1,\dots, f_r\in K[X]^G$, which are homogeneous and algebraically independent. Denote by $F=K[f_1,\dots, f_r]$, then $K[X]^G=Fg_1+\cdots+Fg_s$ for some $g_1,\dots, g_s\in K[X]^G$ homogeneous. Such a structure is also called a \textit{Hironaka decomposition}. The homogeneous polynomials $f_1,\dots, f_r$ are called \textit{primary invariants} and $g_1, \dots, g_s$ are \textit{secondary invariants}. In the remains of the section, we explicitly give a set of fundamental invariants as generators, and a system of primary invariants for $K[V_\Gamma\times V_\Gamma]^{S_\Gamma}$. 
	
	\subsection{Orbit sums in invariant subalgebras}
	To construct a set of generators for $K[V_\Gamma\times V_\Gamma]^{S_\Gamma}$ as a $K$-algebra, we investigate orbit sums of linearly independent generators of $K[V_\Gamma\times V_\Gamma]$ analogous to the elementary symmetric functions. For permutation representations of finite groups as in our case, we are guaranteed to find orbit sum generators of small degrees, see Garsia-Stanton \cite{Garsia-Stanton} or Theorem 6.2.9 of L. Smith \cite{Larry Smith}.
	
	In general, for any finite group $G$ acting on a variety $X$, and $f\in K[X]$, the \textit{orbit sum} of $f$ over $G$ is defined as 
	\[O_G(f):=\sum_{g\in G}g\cdot f.\] 
	(Here the normalizer by $1/|G|$ is safely omitted since the field $K$ has characteristic zero.)
	Clearly any orbit sum belongs to $K[X]^G$ and $O_G(f_1+f_2)=O_G(f_1)+O_G(f_2), \forall f_1, f_2\in K[X]$. Hence to find generators of $K[X]^G$, it suffices to find generators for all orbit sums of monomials over $G$ in $K[X]$. In our case, we need to find generators for orbit sums of $\prod_{i=1}^k\prod_{l=1}^{n_i}x_{il}^{r_{il}}\otimes x_{il}^{s_{il}}$ for any $r_{il}, s_{il}\in\Z$. Note that the multiplication in tensor product of algebras operates as $(f_1\otimes f_2)(g_1\otimes g_2)=f_1g_1\otimes f_2g_2$. Since each $S_{n_i}$ in the direct product $S_\Gamma$ acts on the component $K[V_{n_i}]\otimes K[V_{n_i}]$ of the tensor product $K[V_\Gamma]\otimes K[V_\Gamma]=\otimes_{i=1}^k(K[V_{n_i}]\otimes K[V_{n_i}])$, we immediately have the following rule for separating orbit sums of monomials.
	\begin{lem}\label{lem-orbit sum of monomials separated}
		The orbit sum $O_{S_\Gamma}\left(\prod_{i=1}^k\prod_{l=1}^{n_i}x_{il}^{r_{il}}\otimes x_{il}^{s_{il}}\right)$ is the product of the orbit sums $O_{S_{n_i}}\left(\prod_{l=1}^{n_i}x_{il}^{r_{il}}\otimes x_{il}^{s_{il}}\right)$, multiplied by some constant.
	\end{lem}
    Next we further simplify construction of generators by the following reduction.
    \begin{lem}\label{lem-orbit sum of single tensor}
    	The orbit sum $O_{S_{n_i}}\left(\prod_{l=1}^{n_i}x_{il}^{r_{il}}\otimes x_{il}^{s_{il}}\right)$ over $S_{n_i}$ can be generated by the orbit sums $O_{S_{n_i}}(x_{il}^{r}\otimes x_{il}^{s})$ ($r,s\in\Z$), for $i=1,\dots, k$ and any chosen $1\leq l\leq n_i$.    	
    \end{lem}
    \begin{proof}
    	In the base mixed case ($l_1\neq l_2$), we have \begin{align*}
    	&\dfrac{1}{n_i(n_i-1)}O_{S_{n_i}}\left(x_{il_1}^{r_{1}}\otimes x_{il_1}^{s_{1}}\cdot x_{il_2}^{r_{2}}\otimes x_{il_2}^{s_{2}}\right)=\sum_{\sigma\in S_{n_i}}(x_{i\sigma(l_1)}^{r_1}\otimes x_{i\sigma(l_1)}^{s_1})(x_{i\sigma(l_2)}^{r_2}\otimes x_{i\sigma(l_2)}^{s_2})\\
    	=&\sum_{1\leq m_1\neq m_2\leq n_i}(x_{im_1}^{r_1}\otimes x_{im_1}^{s_1})(x_{im_2}^{r_2}\otimes x_{im_2}^{s_2})\\
    	=&\left(\sum_{1\leq l\leq n_i}x_{il}^{r_1}\otimes x_{il}^{s_1}\right)\left(\sum_{1\leq l\leq n_i}x_{il}^{r_2}\otimes x_{il}^{s_2}\right)-\sum_{1\leq l\leq n_i}x_{il}^{r_1+r_2}\otimes x_{il}^{s_1+s_2}\\
    	=&O_{S_{n_i}}\left(x_{il}^{r_{1}}\otimes x_{il}^{s_{1}}\right)O_{S_{n_i}}\left(x_{il}^{r_{2}}\otimes x_{il}^{s_{2}}\right)-O_{S_{n_i}}\left(x_{il}^{r_{1}+r_2}\otimes x_{il}^{s_{1}+s_2}\right).
    	\end{align*}
    	In general, we have
    	\begin{align*}&O_{S_{n_i}}\left(\prod_{l=1}^{n_i-1}x_{il}^{r_{il}}\otimes x_{il}^{s_{il}}\right)O_{S_{n_i}}\left(x_{in_i}^{r_{in_i}}\otimes x_{in_i}^{s_{in_i}}\right)\\
    	=&c_1\sum_{j=1}^{n_i-1}O_{S_{n_i}}\left(x_{ij}^{r_{ij}+r_{in_i}}\otimes x_{ij}^{s_{ij}+s_{in_i}}\prod_{l\neq n_i,j}x_{il}^{r_{il}}\otimes x_{il}^{s_{il}}\right)+c_2O_{S_{n_i}}\left(\prod_{l=1}^{n_i}x_{il}^{r_{il}}\otimes x_{il}^{s_{il}}\right),
    	\end{align*} 	
    	for some nonzero integer constants $c_1,c_2\in K$. Hence the lemma follows from induction on the number of mixed tensor factors in the orbit sum.
    \end{proof}
    The above results only provide an infinite system of generators. We need to further bound the degree of the generating orbit sums. Note that $\deg(x_{il}^r\otimes x_{jm}^s)=r+s$. As illustration, we demonstrate how symmetric polynomials of $K[x_1,\dots,x_n]$ can be generated by the power sums $s_i=x_1^i+\cdots+x_n^i, i\leq n$. In addition, denote by $e_i=\sum_{I\subset\{1,\dots,n\}, |I|=i}\prod_{j\in I}x_j$ the elementary symmetric polynomials which can be as polynomials in $s_i$'s by Newton's identities. To show that $s_{n+1}$ can be generated per se, we compute
    \begin{align*}s_{n+1}=s_1s_n-\sum_{1\leq i\neq j\leq n}x_ix_j^{n}=&s_1s_n-\left(\left(\sum_{1\leq i\neq j\leq n}x_ix_j\right)s_{n-1}-\sum_{1\leq i\neq j\neq k\leq n}x_ix_jx_k^{n-1}\right)\\
    =&\cdots\\
    =&F(s_1,\dots,s_n)+(-1)^{n-1}\sum_{1\leq i_1\neq\cdots\neq i_n\leq n}x_{i_1}\cdots x_{i_{n-1}}x_{i_n}^2\\
    =&F(s_1,\dots,s_n)+(-1)^{n-1}e_ns_1,
    \end{align*}
    where $F(s_1,\dots,s_n)$ is a polynomial in $s_i$'s. Note that the algebraic dependence emerges at the last step due to the number of variables. This shows that $s_{n+1}$ is generated by $s_1,\dots, s_n$ and is actually an expression of Newton's identities. We extend its use to tensor products. First, parallel to symmetric polynomials we have
    \begin{lem}\label{lem-constant on left or right}
    	The orbit sums $L_{im}:=O_{S_{n_i}}(x_{il}^{m}\otimes 1)$ (for any chosen $1\leq l\leq n_i$) are generated by $L_{im}$ with $m\leq n_i$. Similarly, $R_{im}:=O_{S_{n_i}}(1\otimes x_{il}^{m})$ are generated by $R_{im}$ with $m\leq n_i$. In particular, the $2(n_1+\cdots+n_k)$ orbit sums $L_{im}, R_{im}, i=1,\dots, k, 1\leq m\leq n_i$ are algebraically independent in $K[V_\Gamma\times V_\Gamma]^{S_\Gamma}$.
    \end{lem}
    Now we verify the following algebraic dependence for mixed tensor products. 
    \begin{lem}\label{lem-degree n+1}
    	For any $r,s\in\Z$ with $r+s> n_i$, the orbit sum $O_{S_{n_i}}\left(x_{l}^{r}\otimes x_{l}^{s}\right)$ ($r,s\in\Z$) is generated by those of degree $<r+s$, for each $i=1,\dots,k$ and any chosen $1\leq l\leq n_i$. 
    \end{lem} 
    \begin{proof}
    In short, for any invariants $f$ and $g$, we use $f\sim g$ to denote for $f-cg$ belonging to $K[L_{im},R_{im}]h$ for some invariant polynomial $h$ of degree $<\deg(f)=\deg(g)$ for some constants $c\in K$.
    Lemma \ref{lem-constant on left or right} covers for the unmixed case, so we assume $r,s\in\Z_+$. For any $1\leq r_1\leq r$, we directly compute
    \[O_{S_{n_i}}\left(x_{il}^{r-r_1}\otimes x_{il}^{s}\right)L_{ir_1}=O_{S_{n_i}}\left(x_{il}^r\otimes x_{il}^{s}\right)+O_{S_{n_i}}\left(x_{il_1}^{r-r_1}x_{il_2}^{r_1}\otimes x_{il_1}^{s}\right),\]
    in which $1\leq l_1\neq l_2\leq n_i$.
    This shows that 
    \begin{equation}\label{equation-first mix}O_{S_{n_i}}\left(x_{il}^r\otimes x_{il}^{s}\right)\sim O_{S_{n_i}}\left(x_{il_1}^{r-r_1}x_{il_2}^{r_1}\otimes x_{il_1}^{s}\right).\end{equation}
    For any $1\leq r_2\leq r_1$, we further compute that
    \begin{align*}&O_{S_{n_i}}\left(x_{il_1}^{r-r_1-r_2}x_{il_2}^{r_1}\otimes x_{il_1}^{s}\right)L_{ir_2}\\
    =&O_{S_{n_i}}\left(x_{il_1}^{r-r_1}x_{il_2}^{r_1}\otimes x_{il_1}^{s}\right)+O_{S_{n_i}}\left(x_{il_1}^{r-r_1-r_2}x_{il_2}^{r_1+r_2}\otimes x_{il_1}^{s}\right)\\
    &+O_{S_{n_i}}\left(x_{il_1}^{r-r_1-r_2}x_{il_2}^{r_1}x_{il_3}^{r_2}\otimes x_{il_1}^{s}\right),
    \end{align*}
    where $1\leq l_1\neq l_2\neq l_3\leq n_i$.
    Together with (\ref{equation-first mix}), this shows that
    \[O_{S_{n_i}}\left(x_{il}^r\otimes x_{il}^{s}\right)\sim O_{S_{n_i}}\left(x_{il_1}^{r-r_1-r_2}x_{il_2}^{r_1}x_{il_3}^{r_2}\otimes x_{il_1}^{s}\right).\]
    Note that the right hand side further mixes the tensor product by adding an extra distinct variable. Proceeding with the same algorithm (by induction), we are able to entirely mix the left side of the tensor product and get
    \[O_{S_{n_i}}\left(x_{il}^r\otimes x_{il}^{s}\right)\sim O_{S_{n_i}}\left(x_{il_1}\cdots x_{il_r}\otimes x_{il_1}^{s}\right),\]
    for $1\leq l_1\neq\cdots\neq l_r\leq n_i$, if $r\leq n_i$. For $r\geq n_i$, we would have 
    \[O_{S_{n_i}}\left(x_{il}^r\otimes x_{il}^{s}\right)\sim O_{S_{n_i}}\left(x_{il_1}\cdots x_{il_{n-1}}x_{il_n}^{r-n+1}\otimes x_{il_1}^{s}\right)=O_{S_{n_i}}(x_{i1}\cdots x_{in_{i}})O_{S_{n_i}}(x_{il}^{r-n}\otimes x_{il}^s),\]
    which is already generated as wanted. 
    
    Now suppose $r<n_i$. To mix the right side of the tensor product, we need an intermediate mixture as follows. Based on (\ref{equation-first mix}) we have for any bi-partition $s_1+s_2=s$,
    \begin{align*}O_{S_{n_i}}\left(x_{il_1}^{r-r_1}x_{il_2}^{r_1}\otimes x_{il_1}^{s_1}\right)R_{is_2}=&O_{S_{n_i}}\left(x_{il_1}^{r-r_1}x_{il_2}^{r_1}\otimes x_{il_1}^{s}\right)+O_{S_{n_i}}\left(x_{il_1}^{r-r_1}x_{il_2}^{r_1}\otimes x_{il_1}^{s_1}x_{il_2}^{s_2}\right)\\
    &+O_{S_{n_i}}\left(x_{il_1}^{r-r_1}x_{il_2}^{r_1}\otimes x_{il_1}^{s_1}x_{il_3}^{s_2}\right),
    \end{align*}
    for $1\leq l_1\neq l_2\neq l_3\leq n_i$. The latter two summands on the right are equivalent in the sense of $\sim$ by displaying the product $O_{S_{n_i}}\left(x_{il_1}^{r-r_1}x_{il_2}^{r_1}\otimes x_{il_1}^{s_1}\right)R_{is_2}$. Hence the above equality shows that
    \[O_{S_{n_i}}\left(x_{il}^r\otimes x_{il}^{s}\right)\sim O_{S_{n_i}}\left(x_{il_1}^{r-r_1}x_{il_2}^{r_1}\otimes x_{il_1}^{s}\right)\sim O_{S_{n_i}}\left(x_{il_1}^{r-r_1}x_{il_2}^{r_1}\otimes x_{il_1}^{s_1}x_{il_3}^{s_2}\right).\] 
    Continuing to utilize the strategy we see that for any partitions $r_1+\cdots+r_u=r$ and $s_1+\cdots+s_v=s$, 
    \[O_{S_{n_i}}\left(x_{il}^r\otimes x_{il}^{s}\right)\sim O_{S_{n_i}}\left(x_{il_1}^{r_1}\cdots x_{il_u}^{r_u}\otimes x_{il_1}^{s_1}x_{il_{u+1}}^{s_2}\cdots x_{il_{u+v-1}}^{s_v}\right),\]
    for $1\leq l_1\neq\cdots\neq l_{u+v-1}\leq n_i$. Eventually we are able to entirely mix the tensor product and get
    \begin{align*}O_{S_{n_i}}\left(x_{il}^r\otimes x_{il}^{s}\right)&\sim O_{S_{n_i}}\left(x_{il_1}\cdots x_{il_r}\otimes x_{il_1}x_{il_{r+1}}\cdots x_{il_{n_i}}^{r+s-n_i+1}\right)\\
    &\sim O_{S_{n_i}}\left(x_{i1}\cdots x_{ir}\otimes x_{i(r+1)}\cdots x_{in_i}\right)R_{i(r+s-n_i)}.
    \end{align*}
    Hence again the orbit sum $O_{S_{n_i}}\left(x_{il}^r\otimes x_{il}^{s}\right)$ is indeed generated by those of degree $<r+s$.    
    \end{proof}

    This immediately implies the main result of this subsection as follows.
    \begin{prop}\label{prop-bound degrees of orbit sums}
    	The orbit sums $O_{S_{n_i}}(x_{il}^r\otimes x_{il}^s)$ of degree $\leq n_i$, $i=1,\dots, k$ and for any chosen $1\leq l\leq n_i$, generate $K[V_\Gamma\times V_\Gamma]^{S_\Gamma}$.
    \end{prop}
	\begin{proof}
		First, orbit sums of degree $n_i+1$ are generated by those of degree $\leq n_i$. Then the proposition follows from this base case and induction on degree.
	\end{proof}
    The concise system of generators provided by the proposition gives us straightforward options for primary and some secondary invariants, as shown in details in the following two subsections.	
	\subsection{Primary invariants of the invariant ring}
	For any finite group $G$ acting on a variety $X$, the invariant subalgebra $K[X]^G$ has the same dimension with the coordinate ring $K[X]$. Here dimension is uniformly referred to Krull dimension. The equivalent notion of transcendence degree of affine algebras may be more intuitive for calculation. For example, in our case where $X=V_\Gamma\times V_\Gamma$ is an affine space which may equate with $V_\Gamma\oplus V_\Gamma$, the dimension is $\dim K[V_\Gamma\times V_\Gamma]^{S_\Gamma}=\dim K[V_\Gamma\oplus V_\Gamma]=2(n_1+\cdots+n_k)$. Then by Lemma \ref{lem-constant on left or right} we get 
	\begin{prop}\label{prop-primary invariants}
		The $2(n_1+\cdots+n_k)$ unmixed orbit sums $R_{im}, L_{im}, i=1,\dots, k, m=1,\dots, n_i$ 
		as in Lemma \ref{lem-constant on left or right} serve as a set of primary invariant in the Hironaka decomposition of $K[V_\Gamma\oplus V_\Gamma]^{S_\Gamma}$. 
	\end{prop}
    Apparently any non-singular affine transform of this set of primary invariants serve as another such set. One may construct more such sets using algebraic transforms.
	\subsection{A standard algorithm of verifying linear independence of secondary invariants}\label{subsection-algorithm}
	     
	It would be a much more difficult task to construct a full set of explicit secondary invariants. One may try to employ the method of the remarkable work by Garsia-Stanton \cite{Garsia-Stanton}. Here we only introduce an standard algorithm for computing secondary invariants and use it to verify linear independence of the fundamental invariants we obtained as in previous sections.
	 
	We start with the graded Nakayama Lemma as follows.
	\begin{lem}\label{lem-Nakayama lemma}
		Let $R$ be a graded algebra over $K=R_0$. $M$ a graded $R$-module and $R_+:=\oplus_{d>0}R_d$ the unique maximal homogeneous ideal. Then a subset $S\subset M$ of homogeneous elements generates $M$ as an $R$-module if and only if $S$ generates $M/R_+M$ as a vector space over $K$.  
	\end{lem}
    \begin{proof} See Lemma 3.7.1 of \cite{Derksen-Kemper}. \end{proof}
    The algorithm follows section 3.7.1 of \cite{Derksen-Kemper}. Let $F=K[R_{im}, L_{il}]$ with $i=1,\dots, k, 1\leq l,m\leq n_i$, which is graded inherently from the grading of $K[V_\Gamma\times V_\Gamma]$. By Cohen-Macaulayness we know that $K[V_\Gamma\times V_\Gamma]^{S_\Gamma}$ is a finite free $F$-module, with $S$ as a set of generators by Proposition \ref{prop-bound degrees of orbit sums}. According to the above Nakayama Lemma, to pick a set of secondary invariants, we need to chose a vector basis of $K[V_\Gamma\times V_\Gamma]^{S_\Gamma}/F_+K[V_\Gamma\times V_\Gamma]^{S_\Gamma}$, where $F_+K[V_\Gamma\times V_\Gamma]^{S_\Gamma}$ denotes its maximal homogeneous ideal. Calculation in the invariant sub-algebra is not very convenient since it is still implicit as an $F$-module. So we investigate its structure through the following embedding:
    \[K[V_\Gamma\times V_\Gamma]^{S_\Gamma}/F_+K[V_\Gamma\times V_\Gamma]^{S_\Gamma}\hookrightarrow K[V_\Gamma\times V_\Gamma]/(R_{im}, L_{il})K[V_\Gamma\times V_\Gamma],\]
    where $(R_{im}, L_{il})$ denotes the ideal generated by those unmixed tensors. (To see that this is an embedding, use decomposition of orbit sums into those of monomials.) This tells us that the wanted vector basis have to be linearly independent modulo the ideal $(R_{im}, L_{il})$ in $K[V_\Gamma\times V_\Gamma]$.
    
   Explicit construction of a full set of secondary invariants for $K[V_\Gamma\times V_\Gamma]^{S_\Gamma}$ is beyond the scope of our consideration in this paper, but will be manifested in the next section after the Hilbert series of $K[V_\Gamma\times V_\Gamma]^{S_\Gamma}$ is computed. Here we just establish the following result as a prelude.
    
    \begin{prop}
    	The orbit sums $O_{S_{n_i}}(x_{il}^r\otimes x_{il}^s)\mid 1\leq i\leq k, r,s\in\Z_+, r+s\leq n_i$ (for any chosen $1\leq l\leq n_i$) are all linearly independent invariants modulo the ideal generated by $R_{im}, L_{im}, m=1,\dots, n_i$, hence belong to a full set of secondary invariants of $K[V_\Gamma\times V_\Gamma]^{S_\Gamma}$.
    \end{prop} 
    \begin{proof}
	By the above algorithm, it suffices to show linearly independence of those orbit sums in $K[V_\Gamma\times V_\Gamma]/(R_{im}, L_{il})K[V_\Gamma\times V_\Gamma]$. Suppose  those orbit sums are linearly dependent. Then graded by degree, for each $2\leq d\leq n_i$ we have
	\[\sum_{1\leq r\leq d-1}\alpha_{r,d-r}O_{S_{n_i}}(x_{il}^{r}\otimes x_{il}^{d-r})=\sum_{m=1}^{d}\left(f_{m}O_{S_{n_i}}(x_{il}^m\otimes 1)+g_{m}O_{S_{n_i}}(1\otimes x_{il}^m)\right),\]
	for some $\alpha_{r,s}\in K$ and $f_m, g_m\in K[V_i\times V_i]$ homogeneous with $\deg(f_{m})=\deg(g_{m})=d-m$. The left hand side is invariant under the action of $S_{n_i}$, so are $f_{m}$ and $g_{m}$. Projecting to single variables by setting $x_{i1}=\cdots=x_{i(k-1)}=x_{i(k+1)}=\cdots=x_{in_i}=0$ for each $1\leq k\leq n_i$, we can find all the coefficients of $f_m$ and $g_m$, Then one easily checks that the cross tensors $x_{il_1}^r\otimes x_{il_2}^s$ with $l_1\neq l_2$ on the right can not be killed.
   \end{proof}

    \section{Hilbert series of the invariant subalgebra}    
    For any representation $V$ of a group $G$, define the Hilbert series 
    \[H(K[V]^G,t):=\sum_{d=0}^\infty \dim(K[V]^G_d)t^d,\] 
    where $K[V]_d$ denotes the vector space of polynomials of degree $d$. 
    \subsection{Hilbert series and Hironka decomposition}
    There is a candid encoding of Hironaka decompositions by Hilbert series as follows.
    \begin{prop}[Proposition 6.8.2 of \cite{Larry Smith}]\label{prop-Hilbert series-Hironka decomposition}
    	Suppose $K[V]^G=Fg_1+\cdots Fg_s$ with $F=K[f_1,\dots, f_r]$ is the Hironaka decomposition of $K[V]^G$. Then
    	\[H(K[V]^G,t)=\dfrac{\sum_{j=1}^st^{\deg(g_j)}}{\prod_{i=1}^r(1-t^{\deg(f_i)})}.\]    	
    \end{prop} 
    The Hilbert series, if fairly easier to be determined, would be our guide map for constructing primary and secondary invariants from a known system of generators. The following Molien's formula makes the guide map computationally available in most cases for finite groups.
    
    \begin{prop}[Molien's formula, see Theorem 3.4.2 of \cite{Derksen-Kemper}]\label{prop-Molien's formula}
    	Let $G$ be a finite group and $V$ be a finite dimensional representation over a field $K$ of characteristic not dividing $|G|$. Then 
    	\[H(K[V]^G,t)=\dfrac{1}{|G|}\sum_{g\in G}\dfrac{1}{\det^0_V(1-tg)},\]
    	where for $char(K)=0$, $\det^0_V(1-tg)=(1-t\lambda_1)\cdots(1-t\lambda_n)$ if $\lambda_i, i=1,\dots, n$, are the complex eigenvalues of $g$ as in $\GL(V)$. 
    \end{prop}
    In our case, $X=V\oplus V=\oplus_{i=1}^k V_{i}\oplus V_i$ and $G=S_\Gamma=S_{n_1}\times\cdots\times S_{n_k}$ acts as each $S_{n_i}$ on $V_i\oplus V_i$, we have $\det^0_V(1-t\sigma)=\prod_{i=1}^{k}\det^0_{V_i}(1-t\sigma_i)^2$ for any $\sigma=(\sigma_1,\dots,\sigma_k)\in S_\Gamma$. Hence we get
    \begin{cor}\label{cor-Hilbert series decomposition}
    	For $V$ and $S_\Gamma$ defined as before, we have
    	\[H(K[V_\Gamma\oplus V_\Gamma]^{S_\Gamma},t)=\prod_{i=1}^kH(K[V_i\oplus V_i]^{S_{n_i}},t)=\prod_{i=1}^k\sum_{\sigma_i\in S_{n_i}}\dfrac{1}{\det^0_{V_i}(1-t\sigma_i)^2}.\]
    \end{cor}
    Hence we may focus on the Hilbert series of the components. The characteristic polynomials of permutations in the standard permutation representation has the following simple expression:
    \begin{lem}\label{lem-characteristic polynomial of permutations}
    	For any $\sigma\in S_n$ and $V=\C^n$ the standard representation with $S_n$ permuting its standard basis, we have
    	\[{\rm det}^0_V(1-t\sigma)=(1-t)^{r_1}\cdots(1-t^l)^{r_l},\]
    	if $\sigma$ belongs to the conjugacy class corresponding to the partition of $n$ with $r_i$ parts equal to $i$, for $i=1,\dots, l$.   	
    \end{lem}
    \begin{proof}
	Each $i$-cycle of $\sigma$ affords all the $i$-th roots as eigenvalues, whence a factor $(1-t^i)$ in $\det^0_V(1-t\sigma)$.
    \end{proof}
    To calculate $H(K[V\oplus V]^{S_n},t)$ for the standard representation $V$ of the symmetric group $S_n$, we use Schur functions, as introduced in the next subsection.
    
    \subsection{Schur functions and application}
    We give a brief introduction to Schur functions according to Chapter VI and VII of D. Littlewood \cite{Littlewood}, and first use it to compute $H(K[V]^{S_n},t)$ as a starter example. 
    
    Let $A=(a_{i,j})_{n\times n}$ be any $n$ by $n$ square matrix and $\chi^\lambda$ be the irreducible character of $S_n$ associated to the partition $\lambda$ of $n$. Define the \textit{immanant} of $A$ corresponding to $\lambda$ as 
    \begin{equation}\label{equation-immanant}
    |A|^{(\lambda)}=\sum_{\sigma\in S_n}\chi_\sigma^\lambda a_{1,\sigma(1)}a_{2,\sigma(2)}\cdots a_{n,\sigma(n)},
    \end{equation}  
    where $\chi^\lambda_\sigma$ is the value of $\chi^\lambda$ on $\sigma$.
    For $\lambda=(1^n)$, $\chi^{(1^n)}_\sigma=sgn(\sigma)=\pm1$, hence $|A|^{((1^n))}=|A|$ is the ordinary matrix determinant. General immanants are not multiplicative but they satisfy the basic property of conjugate invariance as determinant does, i.e. $|BAB^{-1}|^{(\lambda)}=|A|^{(\lambda)}$ for any invertible $n$ by $n$ matrix $B$.  
    
    For any $m$ variables $\alpha_1,\cdots,\alpha_m$, let $s_i=\alpha_1^i+\cdots+\alpha_m^i$ be the basic symmetric functions on them. Define a matrix $s=s(\alpha_1,\cdots,\alpha_m)$ of the following shape
    \[s=\begin{pmatrix}
    s_1&1\\
    s_2&s_1&2\\
    &&&\\
    \cdots&\cdots&\cdots&n-1\\
    s_n&s_{n-1}&\cdots&s_1
    \end{pmatrix},\]
    which comes from the Newton's identities $s(-e_1,e_2,\cdots,(-1)^ne_n)^T=0$ with $e_i$ the elementary symmetric functions on $\alpha_1,\cdots,\alpha_m$. Then for any partition $\lambda$ of $n$ we define the \textit{Schur function} $\{\lambda\}=\{\lambda\}(\alpha_1,\dots,\alpha_n)$ as
    \begin{equation}\label{equation-Schur function}
    \{\lambda\}=\dfrac{1}{n!}|s|^{(\lambda)}.
    \end{equation}
    When convergence is not an issue, the Schur functions can be defined for an infinite series of variables. 
    
    We fix some notations of integer partitions. For any partition $\lambda=(1^{r_1}\cdots k^{r_k})$ of $n$, denoted by $|\lambda|=r_1+2r_2+\cdots+ kr_k=n$, define $C_\lambda=1^{r_1}r_1!\cdots k^{r_k}r_k!$, which is the size of the conjugacy class of $S_n$ corresponding to $\lambda$, and $z_\lambda=n!/C_\lambda$, which is the size of the centralizer of any permutation belonging to the conjugacy class in $S_n$. Also define $h_r(\alpha_1,\dots,\alpha_m)=\sum_{u_1+\cdots+u_m=n, u_i\in\Z}\alpha_1^{u_1}\cdots\alpha_m^{u_m}$ as the homogeneous product sums of degree $r$. Then we include two basic results on Schur functions. 
    \begin{lem}[6.2 of \cite{Littlewood}]\label{lem-Schur function identities}
    	For any partition $\lambda=(1^{r_1}2^{r_2}\cdots k^{r_k})$ of $n$, let $s_{\lambda}=s_1^{r_1}s_2^{r_2}\cdots s_k^{r_k}$. Then
    	\[\{\lambda\}=\sum_{|\rho|=n}\dfrac{1}{z_\rho}\chi^{\lambda}_\rho s_\rho,\]
    	summing over all partitions of $n$. In particular, $\{(1^n)\}=e_n(\alpha_1,\dots, \alpha_m)$ and $\{(n)\}=h_n(\alpha_1,\dots,\alpha_m)$ are the $n$-th elementary symmetric polynomial and the homogeneous product sums of degree $n$. On the other hand, we also have
    	\[s_{\lambda}=\sum_{|\mu|=n}\chi_\lambda^\mu\{\mu\}.\]
    \end{lem} 
    The following special case of Schur functions explicitly computable by the Jacobi-Trudi equation (see 6.3 of \cite{Littlewood}) is crucial to our calculation of $H(K[V\oplus V]^{S_n},t)$.
    \begin{lem}[7.1 of \cite{Littlewood}]\label{lem-Schur function on q}
    	For any partition $\lambda=\{\lambda_1,\lambda_2,\cdots,\lambda_l\}$ of $n$ with $\lambda_1\geq \lambda_2\geq\dots\geq \lambda_l$ ($l\leq n$), let $n_\lambda=\lambda_2+2\lambda_3+\cdots+(l-1)\lambda_l$ and $\{\lambda:t\}$ be the Schur function on $\{t^m,m\in\Z\}$ corresponding to $\lambda$, then 
    	\[\{\lambda:t\}=t^{n_\lambda}\dfrac{\prod_{1\leq r<s\leq l}(1-t^{\lambda_r-\lambda_s-r+s})}{\prod_{r=1}^{l}\phi_{\lambda_r+l-r}(t)},\]
    	where $\phi_k(t)=(1-t)\cdots(1-t^k)$ for any positive integer $k$.
    \end{lem}
    
    Now we give a straightforward calculation of $H(K[V]^{S_n},t)$ using Schur functions. First note that for any partition $\lambda=(1^{r_1}\cdots k^{r_k})$ of $n$, the corresponding conjugacy class has size $C_\lambda=1^{r_1}r_1!\cdots k^{r_k}r_k!$. Then by Lemma \ref{lem-characteristic polynomial of permutations} and Molien's formula (Proposition \ref{prop-Molien's formula}), we have
    \begin{align}\label{equation-Hilbert series of Sn}H(K[V]^{S_n},t)&=\dfrac{1}{n!}\sum_{\sigma\in S_n}\dfrac{1}{\det_V^0(1-t\sigma)}=\dfrac{1}{n!}\sum_{|\lambda|=n}\dfrac{C_\lambda}{(1-t)^{r_1}\cdots(1-t^k)^{r_k}}\\
    &=\sum_{|\lambda|=n}\dfrac{1}{z_\lambda}\dfrac{1}{(1-t)^{r_1}\cdots(1-t^k)^{r_k}}.\notag
    \end{align}
    Note that $\chi^{(n)}_\lambda=1$ for any partition $\lambda$ of $n$, and for the infinite series of variables $t^m, m\geq 0$, $s_\lambda=\dfrac{1}{(1-t)^{r_1}\cdots(1-t^k)^{r_k}}$. Hence by Lemma \ref{lem-Schur function identities} and Lemma \ref{lem-Schur function on q}, we have
    \[H(K[V]^{S_n},t)=\{(n)\}(1,t,t^2,\dots)=\dfrac{1}{\phi_n(t)}=\dfrac{1}{(1-t)(1-t^2)\cdots(1-t^n)}.\]
    We can also directly verify the above equality without using Lemma \ref{lem-Schur function on q}. We write
    \[\sum_{|\lambda|=n}\dfrac{1}{z_\lambda}\dfrac{1}{(1-t)^{r_1}\cdots(1-t^k)^{r_k}}=\sum_{u_i\geq 0, \sum_{i\geq 0}u_i=n}\prod_{i\geq 0}t^{iu_i}.\]
    Each term in the summation uniquely corresponds to a partition $(1^{u_1}\cdots i^{u_i}\cdots)$ with $\sum_{i\geq 1}u_i\leq n$ parts, which then uniquely corresponds to its conjugate partition. Rearranging the summation over conjugate partitions enables us to write (\ref{equation-Hilbert series of Sn}) as
    \begin{align*}H(K[V]^{S_n},t)&=\sum_{k_1,\cdots,k_n\in\Z}\prod_{j=1}^nt^{jk_j}=\prod_{j=1}^n(1+t^j+t^{2j}+\cdots)\\
    &=\dfrac{1}{(1-t)(1-t^2)\cdots(1-t^n)}.
    \end{align*}
    \begin{remark}
    The above computation just manifests the fundamental theorem of symmetric functions that $K[V]^{S_n}$ is a polynomial ring in the elementary symmetric polynomials. Moreover, its coincidence with the generating function of numbers of partitions with at most $n$ parts corresponds to the fact that $K[V]_d^{S_n}$ for each $d\geq 0$ has Schur polynomials $\{\lambda\}(x_1,\dots,x_n)$ of degree $d$ as a linear basis, noting that $\{\lambda\}=0$ if the partition $\lambda$ has more than $n$ parts by the Jacobi-Trudi identity. 
    \end{remark}
    To compute $H(K[V\oplus V]^{S_n},t)$, we introduce a method based on orthogonality of characters as follows. By Corollary \ref{cor-Hilbert series decomposition} and Lemma \ref{lem-characteristic polynomial of permutations}, we have
    \begin{equation}\label{equation-Hilbert series of product}H(K[V\oplus V]^{S_n},t)=\sum_{|\rho|=n}\dfrac{1}{z_\rho}\dfrac{1}{\det_V^0(1-t\sigma)^2}=\sum_{|\rho|=n}\dfrac{1}{z_\rho}s_{\rho+\rho}(1,t,t^2,\dots).\end{equation}
    Here for any two partitions $\mu=(1^{r_{1}}\cdots n^{r_{n}})$ and $\nu=(1^{s_{1}}\cdots n^{s_{n}})$ of $n$, we denote by $\mu+\nu$ the partition $(1^{r_{1}+s_{1}}\cdots n^{r_{n}+s_{n}})$ of $2n$. With this notation, we immediately get $s_{\mu}s_{\nu}=s_{\mu+\nu}$. Moreover, by Lemma \ref{lem-Schur function identities}, we compute for any partition $\lambda$ of $n$
    \[\{\lambda:t\}^2=\sum_{|\mu|=|\nu|=n}\dfrac{1}{z_{\mu}z_{\nu}}\chi^{\lambda}_\mu\chi^\lambda_\nu s_{\mu+\nu},\]
    and
    \begin{equation}\label{equation-square Schur function}\sum_{|\lambda|=n}\{\lambda:t\}^2=\sum_{|\lambda|=n}\sum_{|\mu|=|\nu|=n}\dfrac{1}{z_{\mu}z_{\nu}}\chi^{\lambda}_\mu\chi^\lambda_\nu s_{\mu+\nu}=\sum_{|\mu|=|\nu|=n}\dfrac{1}{z_{\mu}z_{\nu}}\left(\sum_{|\lambda|=n}\chi^{\lambda}_\mu\chi^\lambda_\nu\right)s_{\mu+\nu}.\end{equation}
    The above summation can be dramatically simplified by the following orthogonality of characters.
    \begin{lem}\label{lem-orthogonality}
    	For any finite group $G$, let $g,h\in G$ and $Z_g$ denote the centralizer of $g$ in $G$. Then
    	\[\sum_{\chi\in Irr(G)}\chi(g)\bar{\chi(h)}=\begin{cases}
    	|Z_g| & \text{ if } g \text{ is conjugate to }h\\
    	0, & \text{ otherwise}.
    	\end{cases},\]
    	where $Irr(G)$ denotes the set of irreducible characters of $G$.
    \end{lem}
    \begin{proof}
    	For a proof, see \cite[Theorem 3.9]{Etingof}.
    \end{proof}
    Hence (\ref{equation-square Schur function}) and (\ref{equation-Hilbert series of product}) become equal:
    \[\sum_{|\lambda|=n}\{\lambda:t\}^2=\sum_{|\mu|=n}\dfrac{1}{z_{\mu}z_{\mu}}z_\mu s_{\mu+\mu}=H(K[V\oplus V]^{S_n},t).\]
    We summarize the above as follows.
    \begin{prop}\label{prop-Hilbert series of product}
    	Let $V\simeq K^n$ be the standard representation of $S_n$ with $S_n$ permuting a standard basis of $V$, for any field $K$ of characteristic zero, and $\{\lambda:t\}$ the Schur function on $\{t^m, m\in\Z\}$ for any partition $\lambda$ of $n$. Then
    	\[H(K[V\oplus V]^{S_n},t)=\sum_{|\lambda|=n}\{\lambda:t\}^2.\]
    \end{prop}
    	We can generalize the result to arbitrary finite groups using the above argument based on orthogonality of characters as of Lemma \ref{lem-orthogonality}.  	
    	\begin{cor}\label{cor-general Schur function}
    		Let $G$ be a finite group and $W$ a finite dimensional representation of $G$ over a field of characteristic zero. For any irreducible character $\chi$ of $G$, define an analogue of the Schur function by
    		\[S_\chi^W(t):=\dfrac{1}{|G|}\sum_{g\in G}\dfrac{\chi(g)}{\det^0_W(1-tg)}.\]
    		Then
    		\[H(K[W\oplus W]^G, t)=\sum_{\chi} S_\chi^W(t)^2,\]
    		where the summation is over all irreducible characters of $G$. 
    	\end{cor}
    Computational evaluation of the Hilbert series $H(K[V\oplus V],t)$ can be obtained by Lemma \ref{lem-Schur function on q}. We can get another expression of it by further combining Proposition \ref{prop-Hilbert series of product} with the following Cauchy's identity:
    \begin{lem}[see Theorem 38.1 of Bump \cite{Bump}]\label{lem-Cauchy's identity}
    	For any $\alpha_1,\dots, \alpha_k\in\C$ and $\beta_1,\dots,\beta_l\in\C$ with $|\alpha_i|<1, |\beta_j|<1$, we have
    	\[\sum_{|\lambda|=n}\{\lambda\}(\alpha_1,\dots,\alpha_k)\{\lambda\}(\beta_1,\dots,\beta_l)=h_n(\alpha_i\beta_j),\]
    	the homogeneous product sums of degree $n$ on $\alpha_i\beta_j, 1\leq i\leq k, 1\leq j\leq l$.    	
    \end{lem} 
    Using infinite series of variables $\alpha_i=t^i, \beta_j=t^j, i,j\geq 0$, and noting that $t^k$ occurs $k+1$ times for each $k\geq 0$, we immediately get the generating function
    \[\sum_{k\geq 0}h_{k}(\alpha_i\beta_j)x^k=\prod_{i,j\geq 0}(1+\alpha_i\beta_jx+(\alpha_i\beta_jx)^2+\cdots)=\prod_{k\geq 0}(1-t^kx)^{-(k+1)}.\] 
    Further by Carlitz \cite[(5.2)]{Carlitz} or Garsia-Gessel \cite[Theorem 2.3]{Garsia-Gessel}, we have
    \begin{lem}
    \[\prod_{k\geq 0}(1-t^kx)^{-(k+1)}=\prod_{k\geq 0}(1-t^{k+1}(x/t))^{-(k+1)}=\sum_{n\geq 0}\dfrac{x^nf_n(t)}{(1-t)^2\cdots(1-t^n)^2},\]
    where $f_n(t)=\sum_{\sigma\in S_n}t^{m(\sigma)+m(\sigma^{-1})}$ and $m(\sigma)$ is the \textit{major index} of $\sigma$, denoting the sum of all $i$ such that $\sigma(i)>\sigma(i+1)$. 
    \end{lem}
    Hence together with Lemma \ref{lem-Cauchy's identity} we get
    \begin{thm}\label{thm-another expression of Hilbert series}
    	With notations above,
    	\[H(K[V\oplus V]^{S_n},t)=\dfrac{f_n(t)}{(1-t)^2\cdots(1-t^n)^2}.\]
    	
    \end{thm}
    Combining Proposition \ref{prop-Hilbert series of product} and Theorem \ref{thm-another expression of Hilbert series} implies the following
    \begin{cor}\label{cor-an equality}
    	With notations above, 
    	\[f_n(t)=\sum_{\sigma\in S_n}t^{m(\sigma)+m(\sigma^{-1})}=\phi_n^2(t)\sum_{|\lambda|=n}\{\lambda:t\}^2.\]
    \end{cor}
    
    \begin{remark}
    	This equality is also established in Stanley \cite[Theorem 5.1]{Stanley} by combinatorial counting of plane partitions.
    \end{remark}

    The highest degree of secondary invariants of $K[V\oplus V]^{S_n}$ is now apparently $\deg(f_n)=m(\sigma)+m(\sigma)=n(n-1)$ for $\sigma=(1\ n)(2\ n-1)\cdots$, which was first computed in the older account by Carlitz \cite{Carlitz}. Since there are $|S_n|=n!$ terms in the summation of $f_n$, there are clearly $n!$ secondary invariants. We summarize it as 
    \begin{cor}\label{cor-highest degree of secondary}
    	The highest degree of the secondary invariants of $K[V\oplus V]^{S_n}$ is $n(n-1)$, and clearly the number of secondary invariants is $n!$.   	
    \end{cor} 
    \begin{remark}
    	The highest degree of secondary invariants can also be calculated directly from Proposition \ref{prop-Hilbert series of product}. Actually, for $\lambda=(1^n)$, the Schur function $\{\lambda:t\}$ achieves the highest degree $-n$ by the formula of Lemma \ref{lem-Schur function on q}. Hence the highest degree of the secondary invariants is equal to $\deg(\phi^2_n(t))-2n=n(n-1)$.
    \end{remark}
    \begin{remark}
    The sequence OEIS A081285 (see \cite{OEIS}) records a list of coefficients of $f_n$ for $n\leq 40$, which gives us explicitly the degrees of all secondary invariants of $K[V\oplus V]^{S_n}$ for $n\leq 40$. Calculation results for small $n$ using the Schur function formula of Lemma \ref{lem-Schur function on q} indeed coincide with the cited list. We record the first three Hilbert series following the calculation:
    \begin{align*}
    &H(K[V\oplus V]^{S_2},t)=\dfrac{1+t^2}{(1-t)^2(1-t^2)^2};\\
    &H(K[V\oplus V]^{S_3},t)=\dfrac{1+t^2+2t^3+t^4+t^6}{(1-t)^2(1-t^2)^2(1-t^3)^2};\\
    &H(K[V\oplus V]^{S_4},t)=\dfrac{1+t^2+2t^3+4t^4+2t^5+4t^6+2t^7+4t^8+2t^9+t^{10}+t^{12}}{(1-t)^2(1-t^2)^2(1-t^3)^2(1-t^4)^2}.
    \end{align*}
    \end{remark}
    Returning to our starting point on product of symmetric groups, we get
    \begin{cor}
    With notations above, we have
    \[H((K[V_\Gamma]\otimes K[V_\Gamma])^{S_\Gamma}, t)=\dfrac{\prod_{i=1}^kf_{n_i}(t)}{\prod_{i=1}^k\phi_{n_i}^2(t)}.\]
    \end{cor}


\medskip
	\begin{thebibliography}{99}

	    
	    
	    \bibitem{Bona}
	    M. Bona, \textit{Combinatorics of Permutations}, Chapman Hall-CRC, second edition, 2012.
       
        \bibitem{Bump}
        D. Bump, \textit{Lie Groups}, GTM 225, second edition. Springer, 2013.
        
        \bibitem{Carlitz}
        L. Carlitz, The expansion of certain products, \textit{Proc. Amer. Math. Soc.} \textbf{7}, 558-564 (1956).
		
		\bibitem{Derksen-Kemper}
		H. Derksen, G. Kemper, \textit{Computational Invariant Theory}, Encyclopedia of Mathematical Sciences, 1st ed. Springer (2002).
		
		
	
		
		
		\bibitem{Etingof}
		P. I. Etingof et al. \textit{Introduction to representation theory}, Amer. Math. Soc. 59 (2011).
		
		\bibitem{Garsia-Gessel}
		A. M. Garsia, I. Gessel, Permutation statistics and partitions, \textit{Adv. in Math.} \textbf{31}, 288-305 (1979).
		
		
		\bibitem{Garsia-Stanton}
		A. K. Garsia, D. Stanton, Group actions on Stanley-Reisner rings and invariants of permutation groups, \textit{Adv. in Math.} \textbf{51}, 107-201 (1984).
		
		\bibitem{Hochster-Eagon}
		M. Hochster, J. A. Eagon, Cohen-Macaulay rings, invariant theory, and the generic perfection of determinantal loci, \textit{Amer. J. of Math.} \textbf{93}, 1020-1058 (1971).
		
		
		\bibitem{Hochster-Roberts}
		M. Hochster, J. L. Roberts, Rings of invariants of reductive groups acting on regular rings are Cohen-Macaulay, \textit{Adv. in Math.} \textbf{13} (1974), 115-175.
		
		\bibitem{Littlewood}
		D.E. Littlewood, \textit{The theory of group characters and matrix representations of groups}, Oxford, 1940.
		
				
		
		\bibitem{Larry Smith}
		L. Smith, \textit{Polynomial Invariants of Finite Groups}, A. K. Peters, Wellesley, Mass. 1995.
		
		\bibitem{OEIS}
		OEIS Foundation Inc. (2021), \textit{The On-Line Encyclopedia of Integer Sequences}, \href{http://oeis.org/A081285}{http://oeis.org/A081285}.
		
		\bibitem{Stanley}
		R. P. Stanley, The conjugate trace and trace of a plane partition, \textit{J. Combinatorial Theory} \textbf{14} (1973), 53-65.
	\end{thebibliography}
\end{document}